\documentclass[pdf]{article}

\usepackage[english]{babel}

\usepackage[utf8]{inputenc}
\usepackage[utf8]{inputenc}
\usepackage{amsmath}
\usepackage{graphicx}
\usepackage{amssymb}
\usepackage{amsthm}
\usepackage{tikz-cd}
\usepackage{mathrsfs}
\usepackage[colorinlistoftodos]{todonotes}
\usepackage{enumitem}
\usepackage{yfonts}
\usepackage{ dsfont }
\usepackage{MnSymbol}

\title{Collapsing Calabi-Yau fibrations and uniform diameter bounds}

\author{Yang Li}

\date{\today}
\newtheorem{thm}{Theorem}[section]
\newtheorem{lem}[thm]{Lemma}
\newtheorem{cor}[thm]{Corollary}
\newtheorem{prop}[thm]{Proposition}

\theoremstyle{definition}

\newtheorem{eg}[thm]{Example}

\newtheorem{Setting}[thm]{Setting}

\newtheorem*{rmk}{Remark}

\newtheorem*{Acknowledgement}{Acknowledgement}

\newcommand{\ie}{\emph{i.e.} }
\newcommand{\cf}{\emph{cf.} }

\newcommand{\C}{\mathbb{C}}

\newcommand{\norm}[1]{\left\lVert#1\right\rVert}
\newcommand{\Lap}{\Delta}

\DeclareMathOperator{\Tr}{Tr}

\begin{document}
\maketitle

\begin{abstract}
As a sequel to \cite{Licollapsing},
we study Calabi-Yau metrics collapsing  along a holomorphic fibration over a Riemann surface. Assuming at worst canonical singular fibres, we prove a uniform diameter bound for all fibres in the suitable rescaling. This has  consequences on the geometry around the singular fibres.
\end{abstract}

%\chapter{}%

\section{Introduction}

The present paper studies the adiabatic limiting behaviour of Ricci flat K\"{a}hler metrics on a Calabi-Yau manifold under the degeneration of the K\"{a}hler class. The basic setting is:
\begin{Setting}\label{Setting}
Let $(X, \omega_X)$ be an $n$-dimensional projective manifold with nowhere vanishing  holomorphic volume form $\Omega$, normalised to $\int_{X} i^{n^2}\Omega\wedge \overline{\Omega}=1$. Let $\pi: X\to Y$ be a holomorphic fibration onto a Riemann surface, with connected fibres denoted by $X_y$ for $y\in Y$, and without loss of generality $\int_{X_y} \omega_X^{n-1}=1$, and $\int_Y\omega_Y=1$. The singular fibres lie over the discriminant locus $S\subset Y$, and $\pi$ is a submersion over $Y\setminus S$. We assume the singular fibres are normal and have \emph{at worst canonical singularities}. Let $\omega_Y$ be a K\"ahler metric on $Y$, and let $\tilde{\omega}_t$ be the Calabi-Yau metrics on $X$ in the class of $\omega_t=t\omega_X+\pi^*\omega_Y$, for $0<t\ll 1$.
\end{Setting}

\begin{eg}
The most elementary examples are projective Calabi-Yau manifolds with Lefschetz fibrations over $\mathbb{P}^1$, for $n\geq 3$. The basic \emph{non-example} is a K3 surface with an elliptic fibration, such that the singular fibres are of type $I_1$.
\end{eg}

The wider question of collapsing Calabi-Yau metrics is intensely investigated by Tosatti and collaborators \cite{To}\cite{TosattiWeinkoveYang}\cite{TosattiZhang}\cite{GrossToZhang}\cite{GrossToZhang2}\cite{HeinTosatti}. Most of these works concentrate only on what happens away from the singular fibres. The author's previous work \cite{Licollapsing} recognized the importance of the \textbf{uniform fibre diameter bound} for the geometry near the singular fibres. This means
\begin{equation}\label{uniformdiameterbound}
\text{diam}(X_y, t^{-1}\tilde{\omega}_t) \leq C.
\end{equation}
with constants independent of the fibre $X_y$ and the collapsing parameter $t$. More precisely, we mean that any two points on $X_y$ can be joined by some path in $X$ (\emph{not necessarily contained in $X_y$}) whose $t^{-1}\tilde{\omega}_t$-length is uniformly bounded. The central result in \cite{Licollapsing} (modulo some technical generalizations) is essentially

\begin{thm}\label{GHconvergence}
In setting \ref{Setting}, we assume the uniform diameter bound (\ref{uniformdiameterbound}). Fix a singular fibre $X_0$ and a point $P\in X_0$, and let $Z$ be a pointed Gromov-Hausdorff subsequential limit of $(X, t^{-1}\tilde{\omega}_t, P)$. Assuming in addition that any holomorphic vector field on the regular part of $X_0$ vanishes, then $Z$ is isometric to $\bar{X}_0\times \C$ with the product metric, where we equip $\C$ with the Euclidean metric, and $\bar{X}_0$ stands for the metric completion of the singular Calabi-Yau metric on $X_0^{reg}$ in the class $[\omega_X]$. 
\end{thm}

A detailed review of the main steps of \cite{Licollapsing} will be given in section \ref{Outline} (partly because some intermediate conclusions are useful, and partly for technical generalizations). It was also observed in \cite{Licollapsing} that in some special cases the uniform fibre diameter bound can be implied by a conjectural H\"older bound on the K\"ahler potential uniformly on the fibres, and the main evidence in \cite{Licollapsing} is a nontrivial diameter bound for nodal K3 fibres. While this H\"older bound strategy has recently found a number of interesting applications (eg. \cite{Guo}\cite{TianZhang}), the conjecture remains hitherto unresolved, due to the difficulty of complex structure/K\"ahler class degeneration.

This paper is to present a clean uniform proof of 

\begin{thm}\label{uniformdiameterboundthm}
In the setting \ref{Setting}, the uniform fibre diameter bound (\ref{uniformdiameterbound}) holds.
\end{thm}

This combined with Theorem \ref{GHconvergence} has implication on the pointed Gromov-Hausdorff limit around singular fibres.

\begin{rmk}
It should be emphasized that for the uniform fibre diameter bound to hold, the `at worst canonical singular fibre' assumption is \emph{necessary}, at least if $[\omega_X]$ is a rational class. This is because on any smooth fibre $X_y$, the rescaled fibrewise metric $t^{-1}\tilde{\omega}_t$ converges smoothly to the unique Calabi-Yau metric $\omega_{SRF,y}$ on $(X_y, [\omega_X])$ as $t\to 0$, with convergence rate depending on $y\in Y\setminus S$ \cite{To}\cite{TosattiWeinkoveYang}. Thus the uniformity in both $t$ and $y$ will imply a uniform diameter bound for all $\omega_{SRF,y}$ in all $y\in Y\setminus S$, which is known to be equivalent to the `at worst canonical singularity' condition, assuming the rest of setting \ref{Setting}  and in addition that $[\omega_X]$ is an integral class up to a constant multiple \cite{Takayama}. For instance, this uniform fibre diameter bound is \emph{not true} around nodal elliptic curve fibres on a K3 surface.	
\end{rmk}

\begin{rmk}
In the motivating case \cite{Licollapsing} of Calabi-Yau 3-folds with Lefschetz K3 fibrations, the uniform diameter bound and the Gromov-Hausdorff convergence statements are consequences of the author's gluing construction \cite{Ligluing}. It is very plausible that a similar construction can be made for higher dimensional Lefschetz fibrations. But it seems unlikely that a gluing strategy can work in the full generality of at worst canonical singularities.
\end{rmk}

The strategy for the uniform fibre diameter bound has two main new ingredients. The first is a uniform exponential integrability of the distance function on the fibres, which amounts to proving the uniform fibre diameter bound modulo a set of exponentially small measure. This method (\cf Theorem \ref{Uniformexpintegrabilitythm}) is of very general nature and has its independent interest. The second is a judicious application of Bishop-Gromov monotonicity to a critically chosen ball, which prevents a subset of exponentially small measure staying far from the rest of the manifold.

\begin{Acknowledgement}
The author is a 2020 Clay Research Fellow, based at MIT. He thanks Valentino Tosatti for comments.
\end{Acknowledgement}

\section{Outline: from diameter bound to GH limit}\label{Outline}

We now give an outline of Thm. \ref{GHconvergence} largely following \cite{Licollapsing} concerning how to identify the pointed  Gromov-Hausdorff limit of the neighbourhood of the (at worst canonical) singular fibre, in setting \ref{Setting}, assuming the uniform diameter bound (\ref{uniformdiameterbound}). The key is that the uniform diameter bound implies a local non-collapsing condition around any given fibre, which enables the application of many standard geometric analysis arguments, in particular Cheeger-Colding theory.

As useful background facts, 
\begin{prop}\cite{GrossToZhang}\cite{Zhang}\cite{Takayama}\label{atworstcanonical}
Assume the setting \ref{Setting}.
Then
\begin{enumerate}
\item The relative holomorphic volume form $\Omega_y$ defined by $\Omega=\Omega_y\wedge dy$ satisfies the uniform bound
$A_y=\int_{X_y}  i^{(n-1)^2} \Omega_y\wedge \overline{\Omega}_y\leq C$ for all $y$ around any given singular fibre. In fact $A_y$ is continuous in $y$.

\item
The unique Calabi-Yau metrics $\omega_{SRF,y}$ on $X_y$ in the class $[\omega_X]$ have uniformly bounded diameters independent of $y$, or equivalently, these metrics are uniformly volume non-collapsed.

\item
There exists $p>1$ such that
\[
\int_{X_y} |\frac{  \Omega_y\wedge \overline{\Omega}_y }{ \omega_X^{n-1} }|^p \omega_X^{n-1} \leq C.
\]
for all $y$ around a given singular fibre.

\end{enumerate}
Morever, the Calabi-Yau metrics $\omega_{SRF,y}$ are continuous in $y$ in the Gromov-Hausdorff topology, including around singular fibres, where $\omega_{SRF,y}$ is understood as the metric completion of the regular locus for the singular Calabi-Yau metric constructed in \cite{EGZ}.

\end{prop}

\begin{rmk}
The $L^p$ volume bound can be seen by passing to a log resolution. The uniform diameter bound is proved by the technique of \cite{Zhang}, and under the projective class condition it is known to be equivalent to the at worst canonical singular fibre assumption \cite{Takayama}, as an application of Donaldson-Sun theory \cite{DonSun}.
\end{rmk}

\subsection{Basic setup and pointwise estimates}

Write the Calabi-Yau metric in terms of the potential $\phi$ depending on $t$:
\[
\tilde{\omega}_t=\omega_t+\sqrt{-1}\partial\bar{\partial}\phi, \quad \omega_t=t\omega_X+ \pi^* \omega_Y.
\]
The Calabi-Yau condition for $\tilde{\omega}_t$ reads
\begin{equation}\label{CalabiYaucondition}
\tilde{\omega}_t^n=a_t t^{n-1} i^{n^2}\Omega\wedge \overline{\Omega},
\end{equation}
where $a_t$ is a cohomological constant.
Under the normalisation $\int i^{n^2}\Omega\wedge \overline{\Omega}=1$, and $\int_{X_y} [\omega_X]^{n-1}=1$, and since the base is 1-dimensional,
\begin{equation}
a_t=\sum_{k=0}^{1} \pi^*[\omega_Y]^k\cdot[\omega_X]^{n-k} {n\choose k} t^{1-k}.
\end{equation}
In the limit $a_t$ converges to $
a_0=    n \int_Y \omega_Y =n .$

From complex pluripotential theory,
\begin{prop}\label{EGZbound}
	\cite{EGZ2}\cite{De-Pa} There is a uniform constant such that
	$
	\norm{\phi}_{L^\infty} \leq C.
	$
\end{prop}
By a maximum principle argument based on the Chern-Lu formula,
\begin{prop}\label{omega0isbounded} 
	There is a uniform bound $\Tr_{\tilde{\omega}_t}\pi^*\omega_Y \leq C$.
\end{prop}
Consequently, the fibrewise restriction $\tilde{\omega}_t|_{X_y}$ has the pointwise volume density upper bound
\begin{equation}\label{fibrevolumeupperbound}
\frac{\tilde{\omega}_t^{n-1}|_{X_y} }{ \omega_{SRF,y}^{n-1} } =\frac{\tilde{\omega}_t^{n-1}\wedge \omega_Y }{ \omega_{SRF,y}^{n-1}\wedge \omega_Y}
\leq C \frac{\tilde{\omega}_t^n(\Tr_{\tilde{\omega}_t}\pi^*\omega_Y) }{  \Omega_y\wedge \overline{\Omega_y}\wedge \omega_Y} \leq Ct^{n-1}.
\end{equation}
Define the oscillation to be $\text{osc}=\sup-\inf$. By applying Yau's $C^0$-estimate fibrewise, with $\omega_{SFR,y}$ as the background metric (which has uniformly bounded Sobolev and Poincar\'e constants in the `at worst canonical singular fibre' context),
\begin{lem} 
	The fibrewise oscillation satisfies the uniform bound
$
\text{osc}_{X_y} \phi \leq Ct.
$
\end{lem}

Next one introduces the fibrewise average function of $\phi$:
\begin{equation*}
\underline{\phi}=\int_{X_y} \phi \omega_X^2.
\end{equation*}
A computation based on the Chern-Lu inequality gives
\[
\Lap_{\tilde{\omega}_t} (\log \Tr_{\tilde{\omega}_t }\omega_X -\frac{C}{t }(\phi-\underline{\phi}) ) \geq  \Tr_{\tilde{\omega}_t } \omega_X-\frac{\text{Const}}{t}.
\]
Now the fibrewise oscillation bound gives $\frac{1}{t}|\phi-\bar{\phi}|\leq C$, whence a maximum principle argument gives
\begin{thm}\label{lowerboundonmetrictheorem}There is a uniform pointwise lower bound
$
	\tilde{\omega}_t \geq C \omega_t.
$
\end{thm}

The severity of the singularity is measured by the function $H=\frac{\omega_X^{n-1} \wedge \omega_Y}{\omega_X^n}$, whose zero locus is precise the $\pi$-critical points on $X$. By pointwise simultaneous diagonalisation of $\tilde{\omega}_t$ and $\omega_t$,
\begin{cor}\label{upperboundonmetric}
	There is a uniform upper bound
	$
	\tilde{\omega}_t \leq \frac{C}{H}\omega_t.
	$
\end{cor}

In particular, in the subset $\{ H\gtrsim 1  \}\subset X$, namely the region away from the $\pi$-critical points but not necessarily away from the singular fibres, there is a uniform equivalence
\begin{equation}\label{uniformequivalence}
C^{-1}\omega_t \leq \tilde{\omega}_t \leq C\omega_t.
\end{equation}
Around any given point in $ \{ H\gtrsim 1  \}$, Evans-Krylov theory gives that $t^{-1}\tilde{\omega}_t$ has uniform $C^\infty$ bound with respect to the background metric $t^{-1}\omega_t$.

\begin{cor}\label{smoothboundfibrewise}
	Inside $\{H\gtrsim 1\}$,
	\[
	\norm{\nabla_{\omega_X}^{(k)} \frac{1}{t}\tilde{\omega}_t|_{X_y} }_{L^\infty}\leq C(k), \quad  \norm{\nabla_{\omega_X}^{(k)}  (\Tr_ {\tilde{\omega}_t } \omega_Y)  |_{X_y} }_{L^\infty}\leq C(k).
	\]
\end{cor}

\begin{rmk}
It should be emphasized that near the $\pi$-critical points, the metrics $\omega_t$ and $\tilde{\omega}_t$ are far from uniformly equivalent. Furthermore, the pointwise estimate from Cor. \ref{upperboundonmetric} cannot imply the uniform fibre diameter bound (\ref{uniformdiameterbound}), nor do the fibres have any useful lower bound on the Ricci curvature to imply (\ref{uniformdiameterbound}).
Resolving this difficulty is the main concern of the present paper.
\end{rmk}

\subsection{
Local noncollapsing}

From now on we assume (\ref{uniformdiameterbound}) in the exposition.

\begin{prop}
	Assuming   (\ref{uniformdiameterbound}), then $t^{-1}\tilde{\omega}_t$ satisfies the \textbf{local volume non-collapsing} estimate: around any central point $P$, and for any $1 \lesssim R\lesssim t^{-1/2}$,
	\begin{equation}\label{localvolumenoncollapsing1}
	Vol(B_{ t^{-1} \tilde{\omega}_t}(P, R)) \geq CR^{2}.
	\end{equation}
	Morever $Vol(B_{ t^{-1} \tilde{\omega}_t}(P, R)) \geq CR^{2n}$ for any $R\lesssim 1$.
\end{prop}

\begin{proof}
	Assume first that $R\gtrsim 1$. Any fibre contains a subregion $\{H\gtrsim 1  \}$ where $\tilde{\omega}_t$ is uniformly equivalent to $\omega_t$. Thus if $d_{\omega_Y}(y,y')\lesssim Rt^{1/2}/C$, then the $t^{-1}\tilde{\omega}_t$-distance between the two fibres $X_{y}$ and $X_{y'}$ is $O(R)$. Using the fibre diameter bound, we can reach any point on a nearby fibre within $O(R)$ distance, so the ball $B_{t^{-1}\tilde{\omega}_t}(P, CR)$ contains the preimage of $B_{\omega_Y}(\pi(P), Rt^{1/2})$. Since the volume form of $t^{-1}\tilde{\omega}_t$ is $a_t t^{-1} \sqrt{-1}\Omega\wedge \overline{\Omega}$, we obtain the estimate (\ref{localvolumenoncollapsing1}). The $R\lesssim 1$ case follows from Bishop-Gromov monotonicity using the Ricci flatness of $\tilde{\omega}_t$.
\end{proof}

Thus non-collapsing Cheeger-Colding theory applies, and in particular around any point on $X$, including $\pi$-critical points, one can take non-collapsing pointed Gromov-Hausdorff limits of $(X, \frac{1}{t}\tilde{\omega}_t, P)$, with all the standard consequences on its regularity.

\subsection{
Convergence estimates}

Let $t\ll 1$. We fix a central fibre $X_0$, which can be singular. The one-dimensional base condition will be crucially used. Consider a coordinate ball $\{  |y|\leq R \}\subset Y$. Let $\omega_{Y,0}= A_0 \sqrt{-1}dy\wedge d\bar{y}$ be a Euclidean metric on $\{ |y|\leq R \}$, where we recall $A_y=\int_{X_y} i^{(n-1)^2}\Omega_y\wedge \overline{\Omega}_y$.

Chern-Lu inequality gives the subharmonicity
\[
\Lap_{\tilde{\omega}_t} (\log \Tr_{\tilde{\omega}_t }\omega_{Y,0}  ) \geq  0.
\]
Using a slightly tricky argument based on the 3-circle inequality and the Harnack inquality (relying on the local non-collapsing), we deduce

\begin{prop}\label{concentrationestimate}
	Assuming (\ref{uniformdiameterbound}), then we have a \textbf{concentration estimate} for $\Tr_{\tilde{\omega}_t} \omega_{Y,0}$ uniform for all choices of $X_0$:
	\begin{equation}
	\max_{|y|\leq t^{1/2}} \Tr_{\tilde{\omega}_t}  \omega_{Y,0}
	\leq 1+  \frac{C}{|\log t|}  ,
	\end{equation}
	\begin{equation}
	t^{-n}\norm{    \Tr_{\tilde{\omega}_t}\omega_{Y,0} -1       } _{L^1_{\tilde{\omega}_t}(|y|\lesssim t^{1/2})}  \leq \frac{C}{|\log t|}.
	\end{equation}
\end{prop}

The concentration estimate easily entails that the two volume density on $X_0$, given by $(t^{-1}\tilde{\omega}_t)^{n-1}$ and $\omega_{SRF,y}^{n-1}$, are close in the $L^1$-sense. By considering the fibrewise Monge-Amp\`ere equation, one deduces that their relative K\"ahler potential is small in an integral sense. In the regular region $\{ H\gtrsim 1 \}$, this improves the smooth bounds in Cor. \ref{smoothboundfibrewise} to convergence bounds:

\begin{prop}
For any small $\epsilon>0$,
	\begin{equation}\label{fibrewisesmoothconvergence}
	\norm{\nabla^{(k)}_{\omega_X} (\omega_{SRF,y}-\frac{1}{t}   \tilde{\omega}_t|_{X_y} )  }_{L^\infty( {X_y}\cap \{H\gtrsim 1 \}) } \leq \frac{C(k, \epsilon)}{ |\log t|^{1/2-\epsilon}  }.
	\end{equation}
\end{prop}

There is one extra bit of juice one can squeeze out of the Chern-Lu formula and the concentration estimate, using an integration by part argument. We have a  gradient bound, which shows that $d\pi$ is in some sense approximately parallel.
\begin{equation}\label{gradientestimateChernLu}
	\int_{ |y|\lesssim  t^{1/2}    }(\frac{|\nabla d\pi|^2 }{\Tr_{ \tilde{\omega}_t  } \omega_{Y,0}  }-|\partial \log \Tr_{ \tilde{\omega}_t  } \omega_{Y,0} |^2) \tilde{\omega}_t^n \leq \frac{C t^{n-1}} {|\log t|}.
	\end{equation}
All these estimates are indepedent of the choice of $X_0$.

\subsection{Gromov Hausdorff limit around the singular fibre}\label{GHlimitaroundsingularfibresection}

Fix a point $P$ on a (singular) fibre $X_0$, and look at the pointed sequence of Ricci flat spaces $Z_t=(\pi^{-1} B_{\omega_Y} (0, R)\subset X, t^{-1}\tilde{\omega}_t)$. Local noncollapsing implies that after passing to subsequence, there is some complex $n$-dimensional Gromov-Hausdorff limit space $(Z,\omega_\infty)$, with a Hausdorff codimension 4 regular locus $Z^\text{reg}$ which is connected, open, dense, where the limiting metric is smooth.  Morever $Z^\text{reg}$ has a natural limiting complex structure, such that the limiting metric is K\"ahler. We shall suppress below mentions of subsequence to avoid overloading notation, and tacitly understand a Gromov-Hausdorff metric is fixed on the disjoint union $Z_t \sqcup Z   $, which displays the GH convergence. Recall $t\ll 1$.

We wish to identify the complex structure. Some heuristic first: since everything away from the fibre $X_0$ is pushed to infinity by scaling, the limit as a complex variety should be the normal neighbourhood of $X_0$, which is just the trivial product $X_0\times \C$ in the case of a smooth fibre, and the guess is that the same is true for the singular fibre.

More formally, we build comparison maps. Let $u$ denote the standard coordinate on $\C$, and $\omega_\C$ refers to the standard Euclidean metric on $\C$. Define the holomorphic maps \[f_t: (Z_t, t^{-1}\tilde{\omega}_t  ) \to (X\times \C, \omega_X+\omega_{\C} ),
\quad x\mapsto (x, u=t^{-1/2}\pi(x)).
\] 
Our scaling convention is that $\omega_{\C}$ agrees with $t^{-1}\omega_{Y,0}$ under the identification $u=t^{-1/2}y=t^{-1/2}\pi(x)$.

By the uniform bound $\Tr_{\tilde{\omega}_t  } {\omega_t} \leq C $, there is a Lipschitz bound on $f_t$ independent of $t$, so the Gromov-Hausdorff limit inherits a Lipschitz map $f_\infty$ into $X\times \C$. By the interior regularity of holomorphic functions, the limiting map $f_\infty$ is holomorphic. As a rather formal consequence of the uniform fibre diameter bound, we can identify the image:

\begin{lem}
	The image of $f_\infty$ is $X_{0}\times \C$.
\end{lem}

Recall the function $H$ measures the severity of singular effect. Now $H$ is a continuous function on $X_0$, so defines a function on $Z$ by pulling back via $f_\infty$. A qualitative consequence of the regularity in $\{ H\gtrsim 1 \}$ is

\begin{prop}
	The map $f_\infty$ is a biholomorphism  $\{H>0\} \subset Z\to X_{0}^{\text{reg} }\times \C$.
\end{prop}

Local noncollapsing and Ricci-flatness implies 
\[
\liminf \text{Vol}_{ \frac{1}{t} \tilde{\omega}_t    }( \{ |u|\leq D  \} \subset Z_t  ) \geq \text{Vol}_{ \omega_\infty }( \{|u|\leq D \}\subset Z   ).
\]
Using the explicit nature of the Calabi-Yau volume form, and the $C^\infty_{loc}$ convergence over $\{ H>0 \}$, one finds
 
\begin{prop}\label{fullmeasure} (Full measure property)
	The subset $\{ H>0 \}\simeq X_0^{reg}\times \C$ inside $Z$ must have full measure on each cylinder $\{ |u|\leq D  \} \subset Z$, so the set $H=0$ has measure zero in $Z$. In particular $X_0^{reg}\times \C$ is open and dense in $Z$.
\end{prop}

We now study the metric $\omega_\infty$ over the smooth region $X_{0}^{\text{reg} }\times \C$. By passing (\ref{fibrewisesmoothconvergence}) to the limit, and using the continuity of $\omega_{SRF,y}$ at $y=0$,

\begin{prop}
	Over  $X_0^{reg}\times \C$, the limiting metric restricts fibrewise to the Calabi-Yau metric $\omega_{SRF,0}$ on $X_0$.
\end{prop}

We also need information about the horizontal component of the metric. By passing the concentration estimate in Prop. \ref{concentrationestimate} to the limit,

\begin{prop}
	The metric $\omega_\infty$ over $X_0^{reg}\times \C$ satisifies the Riemannian submersion property
	$
	\Tr_{\omega_\infty} \omega_\C=1.
	$
\end{prop}

By passing the gradient estimate (\ref{gradientestimateChernLu}) to the limit,

\begin{prop}
	Over  $X_0^{reg}\times \C$, the differential $du$ is parallel with respect to $\omega_\infty$.
\end{prop}

We can pin down the Riemannian metric on the regular locus:

\begin{prop}\label{limitingmetric}
	The limiting metric $\omega_\infty=\omega_{SRF,0}+\omega_\C$ over  $X_0^{reg}\times \C$.
\end{prop}

\begin{proof}
	The parallel diffential $du$ induces a parallel $(1,0)$ type vector field by the complexified Hamiltonian construction:
	$
	\iota_V \omega_{\infty}=d\bar{u}.
	$
	In particular $V$ is a holomorphic vector field. By assumption, there is no holomorphic vector field on $X_{0}^{\text{reg}}$, so $V$ must lie in the subbundle $T\C\subset T(X_0^{\text{reg} }\times \C)$. Morever, on each fibre $X_{0}^{reg}\times \{u\}$, $V$ must be a constant multiple of $\frac{\partial }{\partial u}$. We can then write $V=\lambda(u)  \frac{\partial }{\partial u}$, where $\lambda$ is a holomorphic function in $u$. Since $du$ and $V$ are both parallel, the quantity $\lambda=du(V)$ must be a constant.

	We know $\omega_\infty$ restricted to the fibres is just $\omega_{SRF,0}$.	
	By construction, the vector field $V$ defines the Hermitian orthogonal complement of the holomorphic tangent space of the fibres. Now $V=\lambda \frac{\partial }{\partial u}$ where the constant is specified by the Riemmanian submersion property. The claim follows.
\end{proof}

\subsection{Geometric convexity}\label{Geometricconvexity}

There is still a small gap between Prop. \ref{limitingmetric} and the Gromov-Hausdorff convergence Theorem \ref{GHconvergence}. By Prop. \ref{limitingmetric}, we know the metric distance on $X_0^{reg}\times \C\subset Z$ is at most that of the product metric. We need to show that this is actually an equality, namely that one cannot shortcut the distance function by going through the singular set in $Z$. (This is the only part of the argument not contained in the more restrictive setting of \cite{Licollapsing}). If so, then the density of $X_0^{reg}\times \C$ in $Z$ (\cf Prop \ref{fullmeasure}) will imply that $Z$ is isometric to $\bar{X}_0\times \C$ as required.

Thus we concentrate on showing

\begin{prop}(Geometric convexity)\label{geometriconvexityprop}
Given two points $P_1, P_2$ in $X_0^{reg}\times \C$, which are GH limits of $P_1^t\in X$ and $P_2^t\in X$ respectively. Then for any given $\epsilon >0$, there is a small enough $\delta$, such that for $t\to 0$, there is a path contained in $\{ H>\delta \}\subset X$ from $P_1^t$ to $P_2^t$, whose $t^{-1}\tilde{\omega}_t$-length is at most $d(P_1^t, P_2^t)+\epsilon$.

\end{prop}

This is precisely what allows one to reduce the distance function computation to knowing the metric only in the regular region. Since $P_1, P_2$ are fixed, we can regard $P_1^t, P_2^t\in \{ H\gtrsim 1 \}$, and $d(P_1^t, P_2^t)\lesssim 1$. It is clear that the question only involves a local region of length scale $O(1)$. The main techniques are developed by Song, Tian and Zhang \cite{Song}\cite{TianZhang}.

The following construction of a good cutoff function is taken from \cite[Lem. 3.7]{Song}, and applied to the singular CY metric $(X_0, \omega_{SRF,0})$.

\begin{lem}
	Given $\lambda>0$ and any compact subset $K$ contained in $X_0^{reg}$.
There is a cutoff function $\rho_\lambda\in C^\infty(X_{reg})$ compactly supported in $X_0^{reg}$,  with $0\leq \rho_\lambda\leq 1$, which equals one on $K$, and satisfies the gradient bound
\[
\int_{X_0} |\nabla \rho_\lambda|^2 \omega_{SRF,0}^{n-1} <\lambda.
\]

\end{lem}

By Cauchy-Schwarz,
\[
\int_{X_0} |\nabla \rho_\lambda| \omega_{SRF,0}^{n-1} \leq C\lambda^{1/2}.
\]
Applying the coarea formula to $|\nabla \rho_\lambda|$ as in \cite[Lem. 2.5]{TianZhang}, we can find a level set $\{ \rho_\lambda=a\}$ compactly contained in $X_0^{reg}\setminus K$, such that
\[
\text{Area}_{ \omega_{SRF,0} }(  \{  \rho_\lambda=a \} ) \leq C\lambda^{1/2}.
\]
Now since $\rho_\lambda$ is supported on the regular locus, we can regard it as a function locally on $X$, which is almost constant in the normal direction to $X_0$. 
Likewise $\{ \rho_\lambda = u\}$ can be regarded as a hypersurface locally on $X$, separating $\{ H\gtrsim 1 \}$ from the most curved region on $X$. For very small $t$ depending on all previous choices,
the metric $t^{-1}\tilde{\omega}_t$ is arbitrarily close to the product metric $\omega_\infty$ on the support of $\rho_\lambda$, whence
\begin{equation}\label{geodesicconvexityarea}
\text{Area}_{ t^{-1}\tilde{\omega}_t }(  \{  \rho_\lambda=a \} \cap d(P,\cdot)\lesssim 1 ) \leq C\lambda^{1/2}.
\end{equation}

\begin{proof}
(Prop \ref{geometriconvexityprop})
 By taking the compact set $K$ large enough, we can ensure $d(P_i^t, \{ \rho_\lambda=a  \} )\gtrsim 1$. The number $\lambda$ can be taken very small depending on $\epsilon$. Suppose there exists a point $Q$ with $d(Q, P_2^t)\lesssim\epsilon$, such that the minimal geodesic from $P_1^t$ to $Q$ does not intersect $\{  \rho_\lambda=a \}\cap \{ d(P,\cdot)\leq 2d(P_1^t, P_2^t) +1 \} $. Then for length reasons this minimal geodesic cannot intersect $\{\rho_\lambda=a  \}$, and since the support of $\rho_\lambda$ is compactly containted in the regular region, this geodesic must stay within $\{ H\gtrsim \delta \}$ for $\delta$ depending only on $\rho_\lambda$, and we can conclude Prop. \ref{geometriconvexityprop}.

Suppose the contrary, namely every minimal geodesic joining $P_1^t$ to any point in $B_{ t^{-1}\tilde{\omega}_t }(P_2^t, \epsilon)$ intersects $\{  \rho_\lambda=a \}\cap \{ d(P,\cdot)\lesssim 1 \} $. By a Bishop-Gromov comparison argument, this would force
\[
\text{Area}_{ t^{-1}\tilde{\omega}_t }(  \{  \rho_\lambda=a \} \cap d(P,\cdot)\lesssim 1 ) \gtrsim \epsilon^{n-1}.
\]
This contradicts (\ref{geodesicconvexityarea}) by taking $\lambda$ small enough in advance.
\end{proof}

\section{Diameter estimates}\label{Diameterestimatessection}

%\textcolor{blue}{Aim: obtain uniform diameter bounds on the fibres near the singularity.}

\subsection{Uniform exponential integrability}

For the moment, we step out of the setting \ref{Setting}, and consider a projective manifold $M\subset \mathbb{P}^N$ of degree $d$ and dimension $n$. Let $\omega_{FS}=\frac{\sqrt{-1}}{2\pi} \log \sum_0^N |Z_i|^2$ be the standard Fubini-Study metric on $(M, c_1(\mathcal{O}(1))$, and $\omega=\omega_{FS}+\sqrt{-1}\partial \bar{\partial}\phi$ be any smooth K\"ahler metric in the same class. The following theorem of independent interest may be regarded as a Riemannian counterpart of uniform Skoda integrability, discussed for instance in \cite{Eleonora} recently.

\begin{thm}\label{Uniformexpintegrabilitythm}
Assume the distance function $d_\omega$ associated to $\omega$ satisfies
\begin{equation}\label{L1distancebound}
\int_{M\times M} d_\omega (y,y') \omega_{FS}^n(y) \omega_{FS}^n(y') \leq A, \quad A\geq 1.
\end{equation}
Then there are constants $C(n)$ depending only on $n$, and $C(n,N,d)$ depending only on $n, N$ and the degree $d$, such that 
\begin{equation}\label{exponentialintegrability}
\int_{M\times M} \omega_{FS}^n(y) \omega_{FS}^n(y') \exp(   \frac{ d_\omega(y,y')  }{C(n)d^2}         ) \leq e^{C(n,N,d)A}.
\end{equation}
\end{thm}

\begin{proof}
Our argument is inspired by Tian and Yau's work on the $\alpha$-invariant \cite{Tianalpha}. As a preliminary discussion, choose a $(N-n-1)$-dimension projective subspace $F\simeq \mathbb{CP}^{N-n-1}\subset \mathbb{CP}^N$, such that $F\cap M= \emptyset$. We project $M$ onto an $n$-dimensional projective subspace $F^\perp$, and call the projection $\pi_F$. (The notation does not suggest perpendicularity for some fixed metric). If $F$ and $F^\perp$ are chosen generically, then $\pi_F: M\to F^\perp$ is a finite map with covering degree $d=\text{deg}(M)$. Let 
\[
\phi_F= \frac{1}{d}\sum_{y\in \pi_F^{-1}(x)  } \phi(y), \quad x\in F^\perp.
\]
Denote by $\psi$ the relative potential between the Fubini-Study metric on $\mathbb{CP}^N$ and $F^\perp$, \ie
\[
\psi_F= \frac{1}{2\pi} \log \frac{ \norm{Z}^2}{ \norm{\pi_F(Z)}^2  }, \quad Z=(Z_0,\ldots, Z_N), \quad [Z_0:\ldots :Z_N]\in \mathbb{CP}^N\setminus F.
\]
Since $F\cap M=\emptyset$, we know $\psi $ is smooth on $M$. %In fact the $L^\infty$ norm is bounded in an open neighbourhood of $M$ inside $\mathbb{CP}^N$, so this constant is uniform for small deformation of $M$, and by a compactness argument constants from this source can be made uniform on the Hilbert scheme.
We observe the pushforward of $\omega$ as a positive (1,1)-current is
\[
\pi_{F*}(\omega)= d\times\omega_{ F^\perp}+ \sqrt{-1} \partial\bar{\partial} ( d\phi_F+  \sum_{ y\in F^\perp} \psi (y) ).
\]
This defines a positive (1,1)-current with continuous potential in $(F^\perp, c_1(\mathcal{O}(d)))$, and is smooth outside the branching locus. By the monotonicity formula in the theory of Lelong numbers, applied to $F^\perp\simeq \mathbb{P}^n$, we have
\begin{equation}\label{Lelong}
\int_{\pi_F^{-1}(B(r))} \omega\wedge \pi_F^*\omega_{F^\perp}^{n-1}=   \int_{B(r)} \pi_{F*} \omega\wedge \omega_{F^\perp}^{n-1} \leq C(n)dr^{2n-2}.
\end{equation}

Now for any $x, x'\in F^\perp$, we consider the function
\[
\rho_F(x,x')= \sum_{ y\in \pi_F^{-1}(x), y'\in \pi_F^{-1}(x')   } d_\omega(y, y'). 
\]
For fixed $x'$, this can be regarded as a function on $x$. Notice $|\nabla_\omega d_\omega(\cdot, y')|\leq 1$ by the definition of distance functions. Thus at least outside the branching locus, we get a pointwise estimate 
\[
\begin{split}
& |\nabla_{\omega_{F^\perp}} \rho_F(x,x')|^2\leq d^2 \sum_{ y\in \pi_F^{-1}(x), y'\in \pi_F^{-1}(x')   } |\nabla_{ \pi_F^*\omega_{F^\perp} } d_\omega(y,y')|^2
\\
& \leq d^2 \sum_{ y\in \pi_F^{-1}(x), y'\in \pi_F^{-1}(x')   }\Tr_{ \omega_{F^\perp} } \omega
\\
& = d^3 \sum_{ y\in \pi_F^{-1}(x)   }\Tr_{ \omega_{F^\perp} } \omega
\end{split}
\]
Here the first inequality uses Cauchy-Schwarz, and the last inequality is because the traces are taken at $y$, with $y'$ fixed.
Since $d_\omega$ comes from a smooth metric on $M$, it is easy to see $\nabla_{\omega_{F^\perp}} \rho_F$ has no distributional term supported on the branching locus. Combining with (\ref{Lelong}), for any fixed $x'$,
\begin{equation}
\int_{B(r)} |\nabla_{F^\perp}\rho_F|^2 \omega_{F^\perp}^n \leq C(n) d^4 r^{2n-2}.
\end{equation}
By the John-Nirenberg inequality,
\begin{equation}\label{JohnNirenberg}
\int_{F^\perp} \exp(  \frac{\rho_F(\cdot, x')-\bar{\rho}_{F, x'}}{ d^2C(n)}   ) \omega_{F^\perp}^n \leq C'(n),
\end{equation}	
where $\bar{\rho}_{F,x'}$ is the average number for fixed $x'$:
\[
\bar{\rho}_{F, x'}= \int_{F^\perp} \rho_F(x,x') \omega_{F^\perp}^n(x).
\]
Define $\bar{\rho}_F=\int_{F^\perp}\int_{F^\perp} \rho_F (x,x')\omega_{F^\perp}^n(x) \omega_{F^\perp}^n(x')$.
Clearly $\bar{\rho}_F$ is the average of $\bar{\rho}_{F,x}$ over all $x\in F^\perp$. The above argument works also for the function $\bar{\rho}_{F,x}$ to give
	\[
	\begin{split}
	\int_{F^\perp} \omega_{F^\perp}^n(x) \exp(   \frac{\bar{\rho}_{F,x}-\bar{\rho}_F}{C(n)d^2}         )     
	\leq C'(n).
	\end{split}
	\]
Now by the change of variable formula,
\[
\begin{split}
& \bar{\rho}_F= \int_{M\times M} d_\omega(y,y') Jac(\pi_F)(y) Jac(\pi_F)(y') \omega_{FS}^n(y)\omega_{FS}^n(y') \\
&\leq 
\int_{M\times M} d_\omega(y,y')  \omega_{FS}^n(y)\omega_{FS}^n(y') \norm{Jac(\pi_F)}_{L^\infty(M)}^2
\\
& \leq A\norm{Jac(\pi_F)}_{L^\infty(M)}^2.
\end{split}
\]
We remark that the $L^\infty$-norm of the Jacobian factor is bounded on $M$ because $M\cap F=\emptyset$, and as long as $M$ is bounded away from $F$ inside $\mathbb{CP}^N$ then this constant stays uniform; this applies to small $C^0$-deformations of $M$, so by the compactness of the Hilbert scheme, such constants can be made uniform for given $n, N, d$ (possibly with changing choices of $F, F^\perp$).

Thus
\begin{equation}
\int_{F^\perp}  \omega_{F^\perp}^n(x) \exp(   \frac{\bar{\rho}_{F,x} }{C(n)d^2}         ) \leq e^{C(n,N,d)A}.
\end{equation}
Combined with Cauchy-Schwarz and (\ref{JohnNirenberg}),
	\begin{equation*}
	\begin{split}
	&\int_{F^\perp\times F^\perp}  \omega_{F^\perp}^n(x) \omega_{F^\perp}^n(x') \exp(   \frac{ \rho_F(x,x')  }{2C(n)d^2}         ) \\
	\leq & \left(\int_{F^\perp\times F^\perp}  \omega_{F^\perp}^n(x) \omega_{F^\perp}^n(x') \exp(  \frac{\rho_F(x,x')-\bar{\rho}_{F,x}}{C(n)d^2}   ) \right)^{1/2}
	\\
	& \times \left(\int_{F^\perp\times F^\perp}  \omega_{F^\perp}^n(x) \omega_{F^\perp}^n(x') \exp (\frac{\bar{\rho}_{F,x} }{C(n)d^2})      \right)^{1/2}
	 \\
	\leq & e^{C(n,N,d)A}.
	\end{split}
	\end{equation*}	
Using the obvious inequality $d_\omega(y,y')\leq \rho_F(\pi_F(y), \pi_F(y') )	$, and changing the value of $C(n)$,
\[
\int_{M\times M} \pi_F^*\omega_{F^\perp}^n(y) \pi_F^*\omega_{F^\perp}^n(y') \exp(   \frac{ d_\omega(y,y')  }{C(n)d^2}         ) \leq e^{C(n,N,d)A}.
\]
Notice this is already very close to our goal (\ref{exponentialintegrability}), in the sense that the exponential integrability
\[
\int \omega_{FS}^n(y) \omega_{FS}^n(y') \exp(   \frac{ d_\omega(y,y')  }{C(n)d^2}         ) \leq e^{C(n,N,d)A}
\]
holds for pairs of points $(y, y')\in M\times M$ where
\begin{equation}\label{genericprojection}
\pi_F^*\omega_{F^\perp}^n \gtrsim \omega_{FS}^n.
\end{equation}
Failure of this essentially means that the differential $d\pi_F$ almost projects the tangent space of $M$ at $y$ or $y'$ to a lower dimensional vector space. Now we recall that the choice of $(F,F^\perp)$ is generic. By varying this choice, we can produce $(F_1, F_1^\perp), \ldots (F_l, F_l^\perp)$, with $l$ suitably large depending on $n, N$, such that for any pair of $(y,y')\in M\times M$, the condition (\ref{genericprojection}) holds for at least one choice of $(F_i, F_i^\perp)$. (For instance, it is enough to take $\{ (F_i, F_i^\perp) \}$ as a suitably dense $\epsilon$-net in the product of Grassmannians.) Morever, the constants are robust for small $C^0$-deformation of $M$ inside $\mathbb{CP}^N$, so by the compactness of the Hilbert scheme again, the constant on the RHS of (\ref{exponentialintegrability}) is uniform in $n, N, d$.
\end{proof}

\begin{rmk}
Some a priori integral bound on $d_\omega$ is necessary, for otherwise $M$ may be disconnected, or degenerating into a union of several components. The same reason shows it is not enough to have an $L^1$-bound on the distance function on a subset of $M\times M$ with say half of the Fubini-Study measure.

However, we claim that it is enough to replace (\ref{L1distancebound}) with an $L^1$-bound on $U\times U$ for a large open subset $U\subset M$ with $(1-\epsilon)$-percent of the Fubini-Study measure, for $\epsilon$ sufficiently small. To see this, first notice that in the John-Nirenberg inequality argument above, we can replace global average on $F^\perp$ by the average on a subset $V$ of $F^\perp$ with say half of the Fubini-Study measure. It is enough to ensure that the $L^1(U\times U)$-bound on $d_\omega$ can bound the $L^1(V\times V)$-norm on $\rho_F$. This amounts to requiring that $U$ contains $\pi_F^{-1}(V)$ for some $V\subset F^\perp$ with half measure, which would be true if $U$ almost carries the full measure.

This remark is quite convenient in situations where one can a priori bound the metric in the generic region of $M$.

\end{rmk}

\begin{rmk}
The above Theorem works for integral K\"ahler classes, but for irrational classes on projective manifolds it is often easy to reduce to the above case.
For instance, consider $M$ a complex submanifold of fixed degree inside a projective manifold $M'\subset \mathbb{CP}^N$. Take an arbitrary fixed K\"ahler class $\chi$ on $M'$, and consider K\"ahler metrics $\omega$ on $M$ in the class $\chi|_M$. We assume on a large enough subset of $M$ that
\[
\int d_{\omega+ \omega_{FS}}(y,y') \omega_{FS}^n(y) \omega_{FS}^n(y') \lesssim 1,
\]
and claim that there exists a uniform bound for all $(M, \omega)$ of the shape
\[
\int_{M\times M} \omega_{FS}^n(y) \omega_{FS}^n(y') \exp(   \frac{ d_\omega(y,y')  }{C  }      ) \leq C',
\]
To see this, we find a large integral multiple $m$, such that $mc_1(\mathcal{O}(1))- \chi$ is a K\"ahler class on $M$, and we choose a K\"ahler representative $\omega'$. Now $\omega'$ is bounded by some constant times $\omega_{FS}$. We can use the Theorem to get an exponential integrability bound for the distance function of $\omega+\omega'$. But it is obvious that distance functions increase with the metric, hence the claim.
\end{rmk}

We can now return to the main setting \ref{Setting}.

\begin{cor}
In the setting \ref{Setting}, there is a uniform exponential integrability bound for all fibres $X_y$ and for $0<t\ll 1$:
\begin{equation*}
\int_{X_y\times X_y}  \exp(  \frac{ d_{ t^{-1}\tilde{\omega}_t }(z,z')   }{C  }       ) \omega_X^{n-1} (z) \omega_X^{n-1} (z') \leq C'.
\end{equation*}
\end{cor}

\begin{proof}
It suffices to prove this for all smooth fibres uniformly. By  Cor. \ref{upperboundonmetric}, the fibrewise metric has an upper bound 
$t^{-1}\tilde{\omega}_t \leq CH^{-1} \omega_X$. Now on any fibre $X_y$, given a prescribed percentage $1-\epsilon$, we can find a subset with at least $1-\epsilon$ of the $\omega_X^{n-1}$-measure, and demand $H$ is bounded below on this subset. Since $\omega_X$ is uniformly equivalent to the Fubini-Study metric, the claim follows from the Remarks above.
\end{proof}

The following Corollary asserts that modulo exponentially small probability, 
any point on $X_y$ is within $O(1)$-distance to the regular region $\{H\gtrsim 1\} \cap X_y$.

\begin{cor}
In the same setting, there are uniform constants such that
\[
\int_{X_y}  \exp(  \frac{ d_{ t^{-1}\tilde{\omega}_t }(z,\{ H\gtrsim 1  \}\cap X_y)   }{C  }       ) \omega_X^{n-1} (z)  \leq C'.
\]
\end{cor}

\begin{proof}
By the Jensen inequality applied to the exp function, using also that $\int_{X_y\cap \{ H\gtrsim 1  \} } \omega_X^{n-1} \geq \frac{1}{2}$,
\[
\begin{split}
LHS \leq & \int_{X_y  }  \exp( C^{-1}\int_{ ( \{ H\gtrsim 1  \}\cap X_y )  }   d_{ t^{-1}\tilde{\omega}_t }(z,z')\omega_X^{n-1}(z')        ) \omega_X^{n-1} (z)  
\\
\leq &  \int_{X_y \times ( \{ H\gtrsim 1  \}\cap X_y ) }  \exp(   \frac{ d_{ t^{-1}\tilde{\omega}_t }(z,z')  }{C  }       ) \omega_X^{n-1} (z) \omega_X^{n-1}(z')
\\
\leq  & C'.
\end{split}
\]
Here $C$ changes from line to line as usual.
\end{proof}

However, what we need is the fibrewise Calabi-Yau volume measure, not some Fubini-Study type measure.

\begin{prop}\label{exponentialintegrability2}
In the same setting, there are uniform constants such that
\[
\int_{X_y}  \exp(  \frac{ d_{ t^{-1}\tilde{\omega}_t }(z,\{ H\gtrsim 1  \}\cap X_y)   }{C  }       ) i^{(n-1)^2} \Omega_y\wedge \overline{\Omega}_y  \leq C'.
\]	
\end{prop}

\begin{proof}
Combine item 3 of Prop. \ref{atworstcanonical} with the above Corollary, and apply H\"older inequality.
\end{proof}

\begin{rmk}
Here we are working with the distance functions on $X_y$ induced by the restriction of $\frac{1}{t} \tilde{\omega}_t$ to $X_y$. We can 
also study the distance function of $\frac{1}{t} \tilde{\omega}_t$ on $X$ and restrict it to $X_y$. This function would be smaller, because the minimal geodesics do not need to be contained in $X_y$. Hence the distance bound can only be better for the latter function, which is what we will use in the next section.
\end{rmk}

\subsection{Uniform fibre diameter bound}

We will now bridge the exponentially small gap between Prop. \ref{exponentialintegrability2} and the uniform fibre diameter bound (\ref{uniformdiameterbound}).

\begin{proof} (Thm. \ref{uniformdiameterboundthm})
Take any point $P$ on $X_y$. All distances appearing below are computed on $X$, not on fibres. Let $r$ be the smallest number such that \[
\text{dist}_{ t^{-1} \tilde{\omega}_t  } ( B_{ t^{-1} \tilde{\omega}_t  }(P, r) , \{  H\gtrsim 1\}\subset X  ) \leq r.
 \]
This exists because the diameter of $X$ is finite (an a priori bound is known but not necessary). If $r\leq 1$, then since $t^{-1}\omega_t$ is uniformly equivalent to $t^{-1}\tilde{\omega}_t$ in $\{ H\gtrsim 1 \}$ (\cf (\ref{uniformequivalence})), we can join $P$ to $\{ H\gtrsim 1  \}\cap X_y$ within $O(1)$-distance, and we are done. So without loss of generality $r\geq 1$. The minimality of $r$ shows that in fact
\begin{equation}\label{criticalball}
\text{dist}_{ t^{-1} \tilde{\omega}_t  } ( B_{ t^{-1} \tilde{\omega}_t  }(P, r) , \{  H\gtrsim 1\}\subset X  ) = r.
\end{equation}
Our strategy is to derive two contrasting bounds on the volume of $B_{t^{-1} \tilde{\omega}_t} (P,r)$.

By Prop. \ref{omega0isbounded}, up to a constant factor the projection $\pi: X\to Y$ decreases distance, so 
\[
\pi (  B_{t^{-1} \tilde{\omega}_t} (P,r)  ) \subset B_{t^{-1}\omega_Y} ( \pi(P), Cr   ) \subset Y.
\]
By Prop. \ref{exponentialintegrability2} and the ensuing Remark,
\[
\begin{split}
& \int_{ \pi^{-1} ( B_{t^{-1}\omega_Y} ( \pi(P), Cr   ) )   } \exp(  \frac{ d_{ t^{-1}\tilde{\omega}_t }(z,\{ H\gtrsim 1  \}\cap X_y)   }{C  }       ) i^{n^2} \Omega\wedge \overline{\Omega}
\\
=
& \int_{ B_{t^{-1}\omega_Y} ( \pi(P), Cr   )  } \sqrt{-1}dy\wedge d\bar{y} \int_{X_y}  \exp(  \frac{ d_{ t^{-1}\tilde{\omega}_t }(z,\{ H\gtrsim 1  \}\cap X_y)   }{C  }       ) i^{(n-1)^2} \Omega_y\wedge \overline{\Omega}_y 
\\
\lesssim &  \int_{B_{t^{-1}\omega_Y} ( \pi(P), Cr   )  } \sqrt{-1}dy\wedge d\bar{y}
\\
\lesssim & r^2t.
\end{split}
\]
But from (\ref{criticalball}), the distance function in the exponent above is bounded below by $r$ on $B_{ t^{-1} \tilde{\omega}_t  }(P, r)$. This forces
\begin{equation*}
\int_{ B_{t^{-1} \tilde{\omega}_t} (P,r)}i^{n^2} \Omega\wedge \overline{\Omega} \lesssim r^2 t e^{-C^{-1} r} ,
\end{equation*}
or equivalently
\begin{equation}\label{criticalballlower}
Vol_{ t^{-1} \tilde{\omega}_t }(   B_{t^{-1} \tilde{\omega}_t} (P,r)     ) \lesssim r^2  e^{-C^{-1} r}
\end{equation}

On the other hand, the ball $B_{t^{-1} \tilde{\omega}_t} (P,2r)$ touches the regular region $\{ H\gtrsim 1 \}$ where $\tilde{\omega}_t$ is uniformly equivalent to $\omega_t$, whence by using the freedom to travel in the regular region,
\[
Vol_{ t^{-1} \tilde{\omega}_t }(   B_{t^{-1} \tilde{\omega}_t} (P,3r)     )\gtrsim r^2.
\]
Since $X$ is Ricci-flat, Bishop-Gromov inequality gives
\begin{equation}\label{criticalballupper}
Vol_{ t^{-1} \tilde{\omega}_t }(   B_{t^{-1} \tilde{\omega}_t} (P,r)     ) \gtrsim r^2 .
\end{equation}
Contrasting (\ref{criticalballlower})(\ref{criticalballupper}) gives $r\lesssim 1$, and we are done.
\end{proof}

\end{document}